\documentclass[a4paper,10pt
]{article}
\pdfoutput=1

\usepackage{balance}
\usepackage[breaklinks]{hyperref}
\usepackage{amssymb}
\usepackage{amsthm}
\usepackage{amsmath}
\usepackage{graphicx}
\usepackage{xcolor}

\long\def\beginpgfgraphicnamed#1#2\endpgfgraphicnamed{\includegraphics{#1}}

\usepackage[low-sup]{subdepth}

\usepackage{sectsty}
\sectionfont{\rmfamily\bfseries\large}
\subsectionfont{\rmfamily\bfseries\normalsize}
\subsubsectionfont{\sffamily\mdseries\large}

\usepackage[T1]{fontenc}
\usepackage[sc]{mathpazo}

\usepackage{textcomp}

\usepackage{ctable}

\makeatletter
\def\operator@font{\sf}
\makeatother

\newtheorem{thm}{Theorem}
\newtheorem*{MaximazingThm}{Leinster's Maximizing Theorem}
\newtheorem*{PosDefThm}{The Positive Definite Subset Bound}
\newtheorem*{cor}{Corollary}

\theoremstyle{definition}

\newcommand{\defn}{\textbf}

\newcommand{\R}{\mathbb{R}}
\newcommand{\CC}{\mathcal{C}}

\DeclareMathOperator{\diam}{diam}
\DeclareMathOperator{\arctanh}{arctanh}
\newcommand{\dd}{\mathrm{d}}

\DeclareMathOperator{\RMSSymbol}{E}
\newcommand{\RMS}{\RMSSymbol_0}
\newcommand{\qRMS}{\RMSSymbol_q}
\newcommand{\RMSq}[1]{\RMSSymbol_{#1}}
\DeclareMathOperator{\LeinCobbSymbol}{D}
\newcommand{\LeinCobb}[3]{{}^{#1}\!\LeinCobbSymbol^{#2}(#3)}
\DeclareMathOperator{\HillSymbol}{Hill}
\newcommand{\Hill}[2]{{}^{#1}\!\HillSymbol({#2})}
\DeclareMathOperator{\Edim}{dim_0}

\DeclareMathOperator{\Vol}{vol}
\DeclareMathOperator{\Area}{area}
\DeclareMathOperator{\TSC}{tsc}

\usepackage{microtype}
\title{Spread: a measure of the size of metric spaces}
\author{Simon Willerton}
\date{} 

\begin{document}

\maketitle

\begin{abstract}
Motivated by Leinster-Cobbold measures of biodiversity, the notion of the spread of a finite metric space is introduced.  This is related to Leinster's magnitude of a metric space.  Spread is generalized to infinite metric spaces equipped with a measure and is calculated for spheres and straight lines.  For Riemannian manifolds the spread is related to the volume and total scalar curvature.  A notion of scale-dependent dimension is introduced and seen, numerically, to be close to the Hausdorff dimension for approximations to certain fractals.
\end{abstract}
\tableofcontents

\section*{Introduction}

Given a finite metric space $X$ with metric $\dd$ we define the \defn{spread} %
$\RMS(X)$ by
  \[\RMS(X):=\sum_{x\in X}\frac{1}{\sum\limits_{x'\in X}e^{-\dd(x,x')}}.\]
This is supposed to be a measure of the size of the finite metric space $X$.  If $X$ has very small distances between all of the points then $X$ looks like a single point and the spread is roughly equal to one.  If $X$ has very large distances  between all of the points then $X$ looks like a collection of very separate points and the spread is roughly equal to the number of points.  In general, of course, a metric space lies between these two extremes and the spread is a measure of how much between these two extremes it is.

The purpose of this paper is to demonstrate some of the basic properties of the spread, to explain the motivation behind its definition and to show how it is connected to other bits of mathematics.  Actually, the spread is one of a family  of metric space `sizes' as we will see in Section~\ref{Section:LeinsterCobbold} where biodiversity motivation is given.
The definition generalizes easily from finite metric spaces to arbitrary metric spaces with a measure, as will be seen in Section~\ref{Section:MeasureSpaces}.

One of the things that we will be interested in is how this measure of size alters as the metric is scaled, so we need to define some notation.  For $t>0$ let $tX$ denote the metric space $X$ with the metric $\dd$ scaled up by a factor of $t$, so that the distance in $tX$ between $x$ and $x'$ is $t\dd(x,x')$.  We can consider the \defn{spread profile} of the space $X$ which is just the graph of $\RMS(tX)$ for $t>0$.

\begin{figure}
\begin{center}   
\beginpgfgraphicnamed{tRspace}
$tR:=\quad
\begin{tikzpicture}[node distance = 4cm, auto,yscale=1.0,baseline=-3
]
 \node [point] (one) at (0,0) {};
    \node [point] (two) at (4.5,0.4) {};
    \node [point] (three) at (4.5,-0.4) {};
  \path [line,<->] (one) -- node[above,pos=0.5] {$1000t$}  (two);
  \path [line,<->] (one) -- node[below,pos=0.5] {$1000t$}  (three);
  \path [line,<->] (three) -- node[right,pos=0.5] {$t$}  (two);
\end{tikzpicture}$
\endpgfgraphicnamed
\\
\beginpgfgraphicnamed{tRspreadprofile}
\begin{tikzpicture}
\begin{axis}[
width = 0.49\textwidth,
axis x line=bottom, axis y line = left,
xmin=-5.2,ymin=0, ymax=3.2, xmax=3,
xtick={-4,-2,...,2}, xticklabels={0.0001,0.01,$1$ ,100},
ytick={0,1,2,3}, yticklabels={0,1,2,3},
x axis line style={style = -},y axis line style={style = -},
xlabel=$t$,
yscale=0.7,
legend style={at={(1,0.1)},anchor=south east}
]
\addplot[mark=none] file {BlockOfDotsMagThreePoints.dat};
\addlegendentry{$\RMS(tR)$}
\end{axis}
\end{tikzpicture}
\endpgfgraphicnamed
\end{center}
\caption{Spread profile of the three-point space~$R$.}
\label{Figure:ThreePointProfile}
\end{figure}
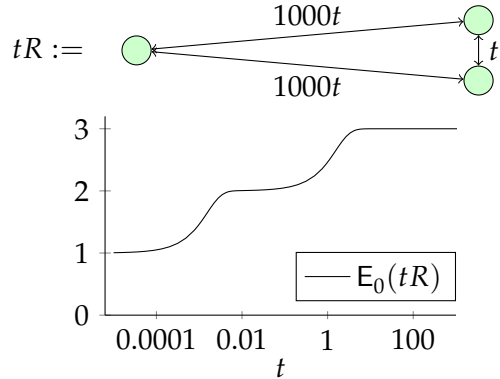

An example of a profile is given in 
Figure~\ref{Figure:ThreePointProfile}.  We consider the space $tR$, for ${t>0}$, having three points, two of which are a distance $t$ apart and are both a distance $1000t$ from the third point.  This family of metric spaces can be thought of as having three `regimes': where $t$ is very small and there looks likes there is one point; where $t$ is smallish and there looks like there are two points; and where $t$ is very large and it looks like there are three points.  This way of thinking is reflected in the values plotted.  So we wish to think of the spread $\RMS$ as akin to an `effective number of points'.

The basic properties of spread, given in the next theorem, all follow easily from the definition.

\begin{thm}
  For $X$ a finite metric space with $N$ points, the spread has the following properties:
  \begin{itemize}
    \item $1\le \RMS(X)\le N$;
    \item  $\RMS(tX)$ is increasing in $t$;
    \item $\RMS(tX)\to 1 $ as $t\to 0$;
    \item $\RMS(tX)\to N $ as $t\to \infty$;
    \item $\RMS(X)\le e^{\diam(X)}$.
  \end{itemize}
\end{thm}

In this paper we consider further properties of the spread which are summarized in the following synopsis.

In the first section we show how the spread can be thought of as the analogue of the number of species in an ecosystem.  We recall Leinster and Cobbold's diversity measures~\cite{LeinsterCobbold:Diversity} and show how that gives rise to the spread $\RMS$ of a metric space as the order-zero diversity of the metric space equipped with the uniform probability distribution.  We also see that there is a spread $\qRMS$ of order $q$ for all $0\le q \le \infty$ and relate these to generalized means.

In the second section we show how the spread relates to Leinster's magnitude~\cite{Leinster:Magnitude} and how the spread can be thought of as being better behaved.  We show that if the space has a positive definite `similarity matrix' then the magnitude is an upper bound for the spread, and if the space is homogeneous then the magnitude is equal to the spread.  We then go on to consider a space with no magnitude, or rather, a space whose magnitude profile is discontinuous and see that the spread profile is very similar, but much better behaved.  Finally in this section we see that two spaces with the same magnitude profile can have different spread profiles.

In the third section we generalize the definition of spread to non-finite metric spaces with a measure, calculate the spread of a straight line interval and relate the asymptotic behaviour of the spread to the volume and total scalar curvature of Riemannian manifolds.  Explicitly, we calculate the spread of $L_\ell$ the straight line interval of length $\ell$ with the usual Lebesgue measure and show that for large $\ell$ the spread is approximately $\ell/2+\ln 2$.  Then we consider the spread of compact Riemannian manifolds, giving the spread of the $n$-sphere explicitly, and we show that asymptotically, as the manifold is scaled up, the leading order terms in the spread are determined by the volume and the total scalar curvature of the manifold.

In the final section we consider the growth rate of the spread, which can be viewed as a kind of (scale dependent) dimension of the space, and we numerically compare this to the Haussdorff dimension for some fractals  Here is some idea of what we mean by scale dependent dimension: if millions of points are formed into the shape of a square, then at small scales it will look like a point, i.e.~zero dimensional, at medium scales it will look two dimensional, and at very large scales it will look like a collection of isolated points, i.e.~zero dimensional, again.  We look at numerical calculations for some simple approximations to fractals and see that at the medium scales the `spread dimension' is related to the Hausdorff dimension of the fractal.

\section{Connection to  Leinster-Cobbold diversity}
\label{Section:LeinsterCobbold}
In this section we recall the notion of Leinster-Cobbold diversity measures and show how this gives rise to the spread of a metric space.  We also see how generalized notions of spread relate to generalized means.
\subsection{Definition of the diversity measure}
In~\cite{LeinsterCobbold:Diversity} Leinster and Cobbold defined certain `diversity measures'.  These are numbers associated to any finite set equipped with a probability distribution and a `similarity matrix' --- we will see that a metric on a finite set gives rise to a similarity matrix in a canonical way.  These numbers are supposed to measure the biodiversity of a community where the points represent the different species, the similarity matrix represents the similarity between the species (a metric represents  distance  between the species) and the probability distribution represents the relative abundances of the species.

Before defining the diversity measures we need to define the notion of a similarity matrix.  If $X$ is a finite set with $N$ points $\{x_1,\dots,x_N\}$ then a \defn{similarity matrix} $Z$ is an $N\times N$ matrix with $0\le Z_{ij}\le 1$ and $Z_{ii}=1$.  If $Z_{ij}=0$ then this represents $x_i$ and $x_j$ being completely dissimilar and if $Z_{ij}=1$ then this represents $x_i$ and $x_j$ being completely identical.  A metric $\dd$ on $X$ gives rise to a similarity matrix $Z$ by setting $Z_{ij}:=\exp(-\dd(x_i,x_j))$, so that nearby points are considered very similar and far away points dissimilar.

Given a finite set $X$ with $N$ points $\{x_1,\dots,x_N\}$ equipped with  a probability distribution $\mathbf{p}=\{p_1,\dots,p_N\}$, so that $\sum_i p_i=1$, and a similarity matrix, $Z$, for $q\in [0,\infty]$, define the \defn{Leinster-Cobbold diversity of order $q$} by
  \[
  \LeinCobb{q}{Z}{\mathbf{p}}:=
  \begin{cases}
    \displaystyle\biggl( \sum_{i:p_i>0} p_i(Z\mathbf{p})^{q-1}_i\biggr)^{\frac{1}{1-q}}
    &q\ne 1,\\[2em]
    \displaystyle\prod_{i:p_i>0} (Z\mathbf{p})_i^{-p_i}
    &q= 1,\\[2em]
    \displaystyle \min_{i:p_i>0} \frac{1}{(Z\mathbf{p})_i}
    &q=\infty.
  \end{cases}
  \]

For fixed $X$, $\mathbf{p}$ and $Z$, the graph of $ \LeinCobb{q}{Z}{\mathbf{p}}$ against $q$ is known as the \defn{diversity profile}.  As a function of $q$, $ \LeinCobb{q}{Z}{\mathbf{p}}$ is monotonically decreasing.  We also have $ 1\le \LeinCobb{q}{Z}{\mathbf{p}}\le N$ and the order $q$ diversity can be thought of as an `effective number of species'.

The Leinster-Cobbold diversity measures generalize a classic family of diversity measures known as the Hill numbers~\cite{Hill:DiversityAndEvenness}.
The Hill number of order $q$, $\Hill{q}{\mathbf{p}}$, for $q\in [0,\infty]$ is defined for a  finite set $X$ with probability distribution  $\mathbf{p}$ on it, so it requires no metric or similarity matrix.   This Hill number can be obtained as the Leinster-Cobbold diversity of the identity similarity matrix, or, equivalenty, of the `discrete' metric where all of the points are infinitely far apart from each other; so all species are considered to be completely dissimilar.  Symbolically, we have
   \[\Hill{q}{\mathbf{p}}=\LeinCobb{q}{I}{\mathbf{p}} .\]
The Hill numbers at the values $q=0,1,2,\infty$ give, respectively, the following classical diversity measures: the number of species, the exponential Shannon index, the Simpson index and the reciprocal Berger-Parker diversity.

We can think of this specialization to Hill numbers as using the Leinster-Cobbold diversity measure to get measures of a finite probability space by equipping the space with a canonical metric, namely the discrete metric.  On the other hand we could use the Leinster-Cobbold diversity measures to get measures of the size of a finite \emph{metric} space by equipping the space with a canonical \emph{probability distribution}, namely the uniform distribution.
This gives rise to the spread.  Symbolically, for a metric space $X$ with $N$ points, define the \defn{$q$-spread} $\qRMS(X)$, for $0\le q\le \infty$, by 
  \[\qRMS(X,\dd):=\LeinCobb{q}{\exp(-\dd)}{(\tfrac1N,\dots,\tfrac1N) }.\]
Explicit formulas are given in Table~\ref{Table:FormulasForEq}.  By the monotonically decreasing nature of the Leinster-Cobbold diversity measures we have that $q\le q'$ implies that $ \qRMS(X)\ge \RMSq{q'}(X)$.  In this paper we have generally concentrated on the greatest of the these values, $\RMS(X)$ which we just call the spread; this is the analogue of the `number of species' in an ecosystem.

\begin{table}[th]
\[\qRMS(X)=\begin{cases}
\displaystyle\Biggl( \frac{1}{N^q}\sum_{i=1}^N \Biggl( \frac{1}{\sum_{j=1}^N Z_{ij}}\Biggr)^{1-q}\Biggr)^{\frac{1}{1-q}}&q\ne 1,\infty,\\[2em]
\displaystyle \sum_{i=1}^N \frac{1}{\sum_{j=1}^N Z_{ij}}&q=0,\\[2em]
\displaystyle N\cdot\prod_{i=1}^N \biggl( \frac{1}{\sum_{j=1}^N Z_{ij}}\biggr)^{1/N}
&q=1, \\[2em]
\displaystyle\frac{N^2}{\sum_{i,j=1}^N Z_{ij}}& q=2,\\[2em]
\displaystyle\min_{i=1,\dots, N} \biggl( \frac{N}{\sum_{j=1}^N Z_{ij}}\biggr)& q=\infty.
\end{cases}\]
\caption{Writing $Z_{ij}:=\exp(-d(x_i,x_j))$, we have these explicit formulas and special cases of the $q$-spread.}
\label{Table:FormulasForEq}
\end{table}

\subsection{Generalized means and reciprocal mean similarity}
Fundamental to the definition of the Leinster-Cobbold diversity measures is the idea of generalized mean~\cite{LeinsterCobbold:Diversity}.  Here we give a description of the $q$-spread in those terms.

Suppose that $X$ is a finite metric space with $N$ points $\{x_1,\dots,x_N\}$, then each point $x_i$ has a \defn{reciprocal mean similarity}  denoted by  $\rho_i$ and defined, as the name suggests, as follows:
  \[\rho_i := \frac{N}{\sum_{j=1}^N e^{-d(x_i,x_j)}}.\]
We have $1\le \rho_i\le N$ and think of the reciprocal mean similarity as being a measure of how different the space is from the point $x_i$, with $\rho_i$ being nearly $1$ if all the points are close to $x_i$ and nearly $N$ if all of the points are far from $x_i$.

In order to get a measure of the whole space we can take an average of these reciprocal mean similarities.  There are many different averages we could take.  For a set of numbers $\mathbf{a}:=\{a_1,\dots, a_N\}$ and a number $s\in \R\cup\{\pm \infty\}$, the $s$-mean $\mu^s(\mathbf{a})$ is defined, when $s\ne 0, \pm \infty$ as
  \[\mu^s(\mathbf{a}):=\biggl( \frac{1}{N}\sum_{i=1}^N a_i^s\biggr)^{1/s},\]
and as a limit when $s= 0, \pm \infty$.
This includes many standard means: $\mu^\infty$ is the maximum, $\mu^2$ is the quadratic mean, $\mu^1$ is the arithmetic mean, $\mu^0$ is the geometric mean, $\mu^{-1}$ is the harmonic mean, and $\mu^{-\infty}$ is the minimum.  These have various nice properties, but the interesting one to note here is that if $s_1>s_2$ then $\mu^{s_1}(\mathbf{a})\ge \mu^{s_2}(\mathbf{a})$ with equality if and only if all of the numbers in $\mathbf{a}$ are equal.

For%
\footnote{The $q$-spread can also be defined for negative $q$, but the properties are slightly different and we do not consider that case here.}
$q\in[0,\infty]$, the $q$-spread $\qRMS(X)$ of the metric space $X$ is by definition the $(1-q)$-mean of the individual reciprocal mean similarities:
  \[\qRMS(X):=\mu^{1-q}(\boldsymbol{\rho}).\]
We have $1\le \qRMS(X)\le N$ with $\qRMS(X)$ being near to $1$ if all of the points are close to each other and $\qRMS(X)$ being near to $N$ if all of the points are far away from each other.

\section{Comparison with magnitude}
In this section we recall Leinster's notion of magnitude and show how it relates to the spread.  We look at examples of a metric space with no magnitude and two metric spaces with the same magnitude.
\subsection{Recap on magnitude}
Magnitude was introduced by Leinster in~\cite{Leinster:Magnitude}.  It is defined for `most' metric spaces in the following way.  For $X$ a metric space a \defn{weighting} on $X$ consists of a \defn{weight} $w_x\in \R$ for each $x\in X$ such that
  \[\sum_{x\in X} w_x e^{-d(x,y)} =1 \qquad \text{for all }y\in X.\]
If a weighting exist then $|X|$ the \defn{magnitude} of $X$ is defined to be the sum of the weights:
  \[|X|:=\sum_{x\in X} w_x.\]
If it exists then the magnitude is independent of any choice in the weighting.  The definition of magnitude comes from enriched category theory, although it had previously appeared in the biodiversity literature~\cite{SolowPolasky:MeasuringBiologicalDiversity}.  For an example of a space without a magnitude see Section~\ref{Section:NaughtyExample} below.  There are large classes of spaces for which the magnitude is known to exist: one class of spaces on which it is defined is the class of `positive definite spaces'.  A \defn{positive definite} finite metric space is a finite metric space for which the similarity matrix $Z$ is positive definite.  Examples of positive definite spaces include subspaces of Euclidean space.  One nice property of the magnitude of positive definite spaces is the following.
\begin{PosDefThm} [{\cite[Corollary~2.4.4]{Leinster:Magnitude}}]
If $X$ is a positive definite space then $|X|$ is well defined, furthermore if $B\subset X$ then $B$ is also positive definite and $|B|\le |X|$.
\end{PosDefThm}
The magnitude is related to the Leinster-Cobbold diversity via the `maximum diversity' $|X|_{+}$.   Before defining that we say that a space $X$ has a \defn{non-negative weighting} if there is a weighting for $X$ in which all of the weights are non-negative.  The \defn{maximum diversity} is defined to be the maximum of the magnitudes of subsets of $X$ with a non-negative weighting:
  \[|X|_{+}:=\max_{B\subseteq X~\text{non-neg}} |B|.\]
For instance, from the Positive Definite Subset Bound it follows that if $X$ is positive definite with non-negative weighting then $|X|_{+}=|X|$.  The connection with diversity is given by the following theorem.
\begin{MaximazingThm}[{\cite[Theorem~3.1]{Leinster:MaximumEntropy}}]
For $X$ a metric space and for any $q\in [0,\infty]$ the maximum value of the Leinster-Cobbold diversity of order $q$, over all probability distributions on $X$, is given by the maximum diversity:
  \[\sup_{\mathbf{p}}\LeinCobb{q}{Z}{\mathbf{p}}=|X|_{+}.\]
\end{MaximazingThm}
This explains the name.  The maximum diversity is certainly in some sense much better behaved than the magnitude, however it is considerably harder to calculate in general.

\subsection{Comparing spread with magnitude}
We can now look at some basic comparisons.

\begin{thm}  Suppose that $X$ is a finite metric space.
\begin{enumerate}
\item The spread of $X$ is bounded-above by the maximum diversity of $X$:
 \[\RMS(X)\le |X|_{+}.\]
\item If $X$ is positive definite then its maximum diversity is bounded above by its magnitude, and thus so is its spread:
  \[\RMS(X)\le |X|_{+}\le |X|.\]
\end{enumerate}
\end{thm}
\begin{proof}
  \begin{enumerate}
  \item This follows from immediately Leinster's Maximizing Theorem and the interpretation of $\RMS(X)$ as the order-zero Leinster-Cobbold diversity of $X$ with the uniform probability distribution.  
  \item \label{Part:two}
  By definition, the maximum diversity of $X$ is the magnitude of a subset $B$ of $X$, so by the Positive Definite Subset Bound, if $X$ is positive definite then $|X|_{+}:=|B|\le |X|$.
  \end{enumerate}
\end{proof}

Note that the positive definite condition in part~\ref{Part:two} of the above theorem cannot simply be removed as we will see in Section~\ref{Section:NaughtyExample} that there is non-positive definite space with magnitude smaller than spread.

We can show that the spread is actually equal to the magnitude in the special case of a homogeneous metric space.  Recall that a homogeneous space is a space in which the points are all indistinguishable, or, more precisely, a homogeneous metric space is a space with a transitive action by a group of isometries.
\begin{thm} If $X$ is a homogeneous finite metric space then the magnitude and the spread coincide:
 \[\RMS(X)=|X|.\]
More generally, the magnitude is equal to the $q$-spread for all $q\in [0,\infty]$:
 \[\qRMS(X)=|X|.\]
\end{thm}
\begin{proof}
If $X$ has $N$ points, then the Speyer's Formula~\cite[Theorem~1]{LeinsterWillerton:AsymptoticMagnitude} for the magnitude of a homogeneous space, we have for any $x\in X$ that
 \[|X|=\frac{N}{\sum_{x'\in X} e^{-\dd(x,x')}}.\]
On the other hand, every point in $X$ has the same mean reciprocal similarity $\rho$, with 
\[\rho=\frac{N}{\sum_{x'\in X} e^{-\dd(x,x')}}.\]
The $q$-spread  $\qRMS(X)$ is just the $(1-q)$-mean of the set of mean reciprocal similarities.  But the $(1-q)$-mean of $N$ copies of $\rho$ is just $\rho$, thus 
 \[\qRMS(X)=\rho=|X|,\]
 as required.
\end{proof}
There is slightly interesting notational coincidence when $Z$ is invertible:
  \[
  \RMS(X)=\sum_{i=1}^N  \biggl({\sum_{j=1}^N Z_{ij}}\biggr)^{-1};\quad 
  |X|=\sum_{i=1}^N  {\sum_{j=1}^N \left(Z^{-1}\right)_{ij}}~.
  \]

\subsection{A space with no magnitude}
\label{Section:NaughtyExample}
Here we look at an example, given by Leinster in~\cite{Leinster:Magnitude}, which has a discontinuity in its magnitude profile and look at its continuous spread profile.  Consider the five-point space $K_{3,2}$ illustrated in Figure~\ref{Figure:BadMagnitude}, equipped with the metric induced by the pictured graph, so that points on the same side are a distance $2$ from each other and points on opposite sides are a distance $1$ from each other.  As we scale this space, considering $tK_{3,2}$ for $t>0$, we see that when $t=\ln(2^{1/2})$ this has no magnitude.  However, the spread $\RMS( tK_{3,2})$ is defined for all values of $t>0$, and this seems to be a well-behaved version of the magnitude.  This example also shows that magnitude is not always an upper bound for the spread. 
%
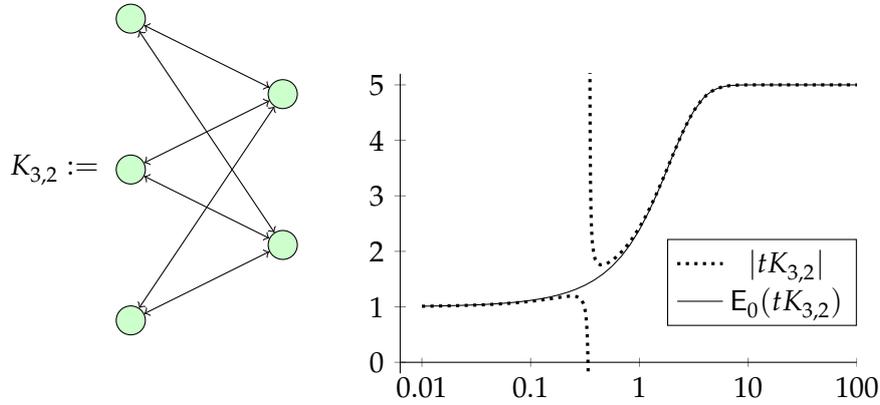
\begin{figure*}[th]
\begin{center}   
\beginpgfgraphicnamed{naughtyexampleA}
\begin{tikzpicture}[node distance = 4cm, auto,baseline=-3em]
    \node (K) at (-1,2) {$K_{3,2}:=$};
 \node [point] (A) at (0,0) {};
    \node [point] (B) at (0,2) {};
    \node [point] (C) at (0,4) {};
    \node [point] (D) at (2,1) {};
    \node [point] (E) at (2,3) {};
  \path [line,<->] (A) -- node[above,pos=0.5] {}  (D);
  \path [line,<->] (A) -- node[above,pos=0.5] {}  (E);
  \path [line,<->] (B) -- node[above,pos=0.5] {}  (D);
  \path [line,<->] (B) -- node[above,pos=0.5] {}  (E);
  \path [line,<->] (C) -- node[above,pos=0.5] {}  (D);
  \path [line,<->] (C) -- node[above,pos=0.5] {}  (E);
\end{tikzpicture}
\endpgfgraphicnamed
\qquad
\beginpgfgraphicnamed{naughtyexample}
\begin{tikzpicture}
\begin{axis}[
width = 0.6\textwidth,
axis x line=center, axis y line = left,
xmin=-2.2,ymin=-0.2, ymax=5.2, xmax=2,
xtick={-2,...,2}, xticklabels={0.01,0.1,$1$,10 ,100},
ytick={0,1,2,3,4,5}, yticklabels={0,1,2,3,4,5},
x axis line style={style = -},y axis line style={style = -},
xlabel=$t$,
yscale=0.8,
legend style={at={(1,0.3)},anchor=south east}
]
\addplot[mark=none,black, very thick, dotted] file {BlockOfDotsMagBadFivePointsI.dat};
\addlegendentry{$|tK_{3,2}|$};
{\addplot[mark=none,black] file {BlockOfDotsE0BadFivePoints.dat};}
\addlegendentry{$\RMS( tK_{3,2})$};
\addplot[mark=none,black, very thick, dotted] 
file {BlockOfDotsMagBadFivePointsII.dat};
\end{axis}
\end{tikzpicture}
\endpgfgraphicnamed
\end{center}
\caption{A five-point space $K_{3,2}$ with the plot of the singularity in its magnitude profile together with its spread profile.}
\label{Figure:BadMagnitude}
\end{figure*}

It is straightforward to generate examples of metric spaces with many points on the magnitude profile not defined, for instance you can use maple to take a random graph with say $100$ vertices and an expected valency of $10$ at each vertex.

We summarize here some of the good properties that the spread has when compared with the magnitude.
\begin{itemize}
\item The spread $\RMS$ is defined for all metric spaces.
\item As an $N$-point space is scaled up, the spread $E_0$ increases from $1$ to $N$.
\item It is much easier to calculate $\RMS(X)$ than $|X|$.
\end{itemize}

%

\subsection{Trees with the same magnitude}
A further class of metric spaces to consider is that of trees.  Given a tree, that is a graph with no cycles, we get a metric space consisting of the set of vertices and the edge-length.  We can generalize an example of Leinster~\cite[Example~2.3.5]{Leinster:Magnitude} to show that all trees with the same number of vertices have the same magnitude.
\begin{thm}
Suppose that  $T_N$ is a tree with $N$ vertices for $N\ge 1$, then the magnitude function is given by
  \[|tT_N|=\frac {N(e^{t}-1)+2}{e^{t}+1}.\]
\end{thm}
\begin{proof}
Observe that if $N=1$ then $|tT_1|=1$ and the result holds.  Suppose that $N>1$.  Pick a leaf $v$, i.e.~a univalent vertex, of the tree $T_N$.  Let $A$ be the metric space $T_N\setminus v$ and let $B$ be the submetric space of $T_N$ consisting of $v$ and its adjacent vertex.  Then by~\cite[Corollary~2.3.3]{Leinster:Magnitude} we have
  \begin{align*}
  |tT_N|&=|tA|+|tB|-1=|tA|+\frac{2}{1+e^{-t}}-1
  \\
  &=
  |tA|+\frac{e^{t}-1}{e^{t}+1},
  \end{align*}
and as $A$ is a tree with $N-1$ vertices the result follows by induction.
\end{proof}

Let's look in particular at two extreme examples; these are pictured in Figure~\ref{Fig:LinearAndCorona}.  On the one hand we have $L_N$ the \defn{linear tree} with $N$ vertices; on the other hand we have $C_N$ the \defn{corona} with $N$ vertices, that is the tree with one ``central'' vertex which has an edge to each of the other vertices, and there are no other edges.  The corona $C_N$ can be thought of as the complete bipartite graph $K_{1,N-1}$.  Both of these $N$-trees give positive-definite metric spaces.  By the above theorem these two spaces have the same magnitude function.
However, they have various different properties.  For instance, their diameters are distinct, $\diam(L_N)=N-1$ and $\diam(C_N)=2$; and $tL_N$ always has a positive weighting, whereas $C_N$ has a negative weight on the central point if $N\ge 5$ (and $tC_4$ has a negative weight on the central point if $t<\ln(2)$).

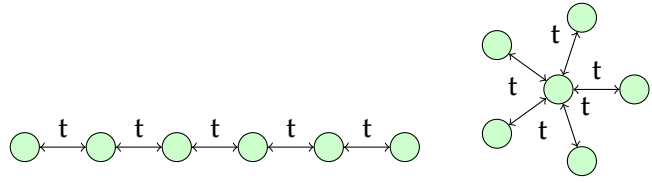
\begin{figure}
\begin{center}   
\beginpgfgraphicnamed{LinearTree}
\begin{tikzpicture}[node distance = 4cm, auto,baseline=-1em]
 \node [point] (A) at (0,0) {};
    \node [point] (B) at (1,0) {};
    \node [point] (C) at (2,0) {};
    \node [point] (D) at (3,0) {};
    \node [point] (E) at (4,0) {};
    \node [point] (F) at (5,0) {};
  \path [line,<->] (A) -- node[above,pos=0.5] {t}  (B);
  \path [line,<->] (B) -- node[above,pos=0.5]  {t} (C);
  \path [line,<->] (C) -- node[above,pos=0.5]  {t} (D);
  \path [line,<->] (D) -- node[above,pos=0.5]  {t} (E);
  \path [line,<->] (E) -- node[above,pos=0.5]  {t} (F);
\end{tikzpicture}
\endpgfgraphicnamed
\qquad
\beginpgfgraphicnamed{Corona}
\begin{tikzpicture}[node distance = 4cm, auto,baseline=-2em]
 \node [point] (M) at (0,0) {};
    \node [point] (N) at (0:1) {};
    \node [point] (O) at (72:1) {};
    \node [point] (P) at (144:1) {};
    \node [point] (R) at (-144:1) {};
    \node [point] (S) at (-72:1) {};
  \path [line,<->] (M) -- node {t}  (N);
  \path [line,<->] (M) -- node {t}  (O);
  \path [line,<->] (M) -- node{t}  (P);
  \path [line,<->] (M) -- node{t}  (S);
  \path [line,<->] (M) -- node {t}  (R);
\end{tikzpicture}
\endpgfgraphicnamed
\caption{The linear tree $tL_{6}$ and the corona $tC_6$ with six points.}
\label{Fig:LinearAndCorona}
\end{center}
\end{figure}

The spread distinguishes these spaces.  An easy calculation gives the following.
 \begin{align*}
   \RMS(tL_N)&=\sum_{i=1}^N \frac{e^t-1}{1+e^t-e^{-t(i-1)}-e^{-t(N-i)}}
   \\
   \RMS(tC_N)&=\frac{1}{1+(N-1)e^{-t}} +  \frac{N-1}{1+e^{-t}+(N-2)e^{-2t}}
 \end{align*}
 It is not too hard to calculate the maximum diversity function either.  The linear tree is positive definite and has a positive weighting, therefore the maximum diversity is precisely the magnitude.  The corona does not always have a positive weighting, and one finds that the central point needs to be `switched off' when the corona is scaled down sufficiently; this gives the following maximum diversity function.
  \[
    |tC_N|_{+}=
     \begin{cases}
      \dfrac {N(e^{t}-1)+2}{e^{t}+1}&t\ge\ln(N-2)\\
      \dfrac {N-1}{1+(N-2)e^{-2t}}&t<\ln(N-2)
     \end{cases}
  \]
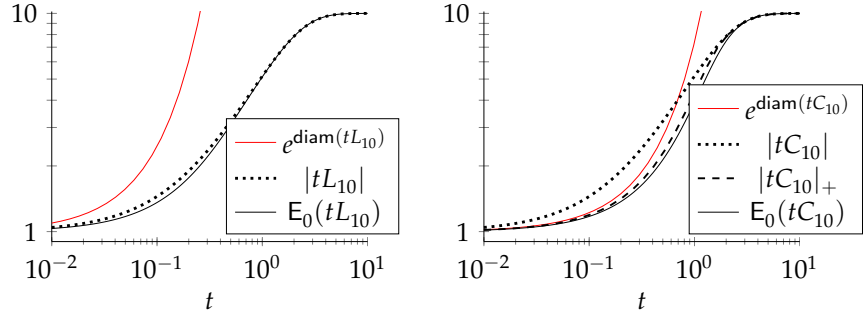
\begin{figure*}
\begin{center}
\beginpgfgraphicnamed{TreeGraphI}
\begin{tikzpicture}[scale=0.9]
\begin{loglogaxis}[
width = 0.49\textwidth,
axis x line=bottom, axis y line = left,
xmin=0.01,ymin=0.9, ymax=10.2, xmax=10,
ytick={1,...,10}, yticklabels={1,,,,,,,,,10},
x axis line style={style = -},y axis line style={style = -},
xlabel={$t$}, 
yscale=0.9,
legend style={at={(1.1,0.03)},anchor=south east}
]
\addplot[mark=none,red,very thin] expression [domain=0.01:1.2]{exp(9*x)};
\addlegendentry{$e^{\diam(tL_{10})}$};
\addplot[mark=none,black,dotted, very thick] [samples=100,domain=0.01:10] expression {(2+10*(exp(x)-1))/(exp(x)+1)};
\addlegendentry{$|tL_{10}|$};
\addplot[mark=none,black] file {E0LinearTree10.dat};
\addlegendentry{$\RMS(tL_{10})$};
\end{loglogaxis}
\end{tikzpicture}
\endpgfgraphicnamed%
\quad%
\beginpgfgraphicnamed{TreeGraphII}
\begin{tikzpicture}[scale=0.9]
\begin{loglogaxis}[
width = 0.49\textwidth,
axis x line=bottom, axis y line = left,
xmin=0.01,ymin=0.9, ymax=10.2, xmax=10,
ytick={1,...,10}, yticklabels={1,,,,,,,,,10},
x axis line style={style = -},y axis line style={style = -},
xlabel={$t$}, 
yscale=0.9,
legend style={at={(1.2,0.03)},anchor=south east}
]
\addplot[mark=none,red,very thin] expression [domain=0.01:1.2]{exp(2*x)};
\addlegendentry{$e^{\diam(tC_{10})}$};
\addplot[mark=none,black,dotted, very thick] [samples=100,domain=0.01:10] expression {(2+10*(exp(x)-1))/(exp(x)+1)};
\addlegendentry{$|tC_{10}|$};
\addplot[mark=none,black,dashed,thick] [samples=24,domain=0.01:2.08] expression {(9)/(1+(8)*exp(-2*(x)))};
\addlegendentry{$|tC_{10}|_{+}$};
\addplot[mark=none,black] expression [samples=100,domain=0.01:10] {1/(1+(10-1)*exp(-x))+ (10-1)/(1+exp(-x)+(10-2)*exp(-2*(x)))};
\addlegendentry{$\RMS(tC_{10})$};
\addplot[mark=none,black,dashed,thick] [samples=6,domain=2.08:10] expression {(2+10*(exp(x)-1))/(exp(x)+1)};
\end{loglogaxis}
\end{tikzpicture}
\endpgfgraphicnamed
\end{center}
\caption{Profiles for the linear tree $L_{10}$ and the corona $C_{10}$ with $10$ points.  The magnitude profile --- the dotted line --- is the same for both spaces.  The exponential in the diameter bounds the maximum diversity and the spread.}
\label{Figure:Corona}
\end{figure*}

The linear tree and the corona with $10$ points are compared in Figure~\ref{Figure:Corona}.  Whilst they have the same magnitude they clearly do not have the same maximum diversity nor spread: the linear tree has greater spread than the corona.  The magnitude of both spaces grows essentially linearly to start with (see Section~\ref{Section:GrowthRate}), which does not reflect the compact nature oft the corona.  The spread (and maximum diversity) grows linearly for the linear tree and exponentially for the corona, reflecting the geometry of these spaces somewhat more.

\section{Generalization to non-finite metric spaces}
\label{Section:MeasureSpaces}
The spread of a finite metric space was defined by using the canonical uniform probability measure on the underlying finite set.  The definition generalizes immediately to any metric space equipped with a finite mass measure.  If $(X,\dd)$ is a metric space equipped with a measure $\mu$ such that $\mu(X)<\infty$, then we can define the spread of $X$ by
 \[\RMS(X):=\int_{x\in X} \frac{\mathrm{d}\mu(x)}{\int_{y\in X} e^{-\dd(x,y)} \,\mathrm{d}\mu(y)}.\]
 This is really the spread with respect to the associated probability measure $\mu/\mu(X)$, but the two factors of $\mu(X)$ cancel in the numerator and denominator.  For $\qRMS(X)$ with $q>0$ the total mass $\mu(X)$  makes an appearance in the definition.
 
 We can now look at the following examples: the line interval with the Lebesgue measure; the $n$-sphere with its intrinsic metric and standard measure; and, asymptotically, any compact Riemannian manifold.

\subsection{The closed line interval}
We can quite straightforwardly calculate the spread of the length $\ell$ line interval $L_\ell$ equipped with the standard Lebesgue measure.
\begin{thm}
  We have 
    \[\RMS(L_\ell) = \frac{\arctanh(\sqrt{1-e^{-\ell}})}{\sqrt{1-e^{-\ell}}} ,\]
   and asymptotically, as $\ell\to \infty$,
   \[\RMS(L_\ell)-(\ell/2 +\ln(2))\to 0.\]
\end{thm}
\makeatletter
\newcommand{\vast}{\bBigg@{3}}
\newcommand{\Vast}{\bBigg@{5}}
\makeatother

\begin{proof}
  \newcommand{\ee}[1]{e^{#1}}
  This is just a case of calculating the integral.  First observe that for $x\in [0,\ell]$,
  \begin{align*}
    \int_{y \in L_\ell} \ee{-\dd(x,y)}dy
    &=\int_{y =0}^{\ell} \ee{-\left|x-y\right|} dy
    =     \int_{y =0}^{x} \ee{-x+y} dy +\int_{y =x}^{\ell} \ee{-y+x} dy
    \\
    &=  \left[\ee{-x+y} \right]_{y =0}^{x}  - \left[\ee{-y+x}\right]_{y =x}^{\ell}
    = 2- \left(\ee{-x}+\ee{-(\ell-x)}\right).
   \end{align*} 
   Thus
   \begin{align*}
     \RMS(L_\ell)
     &=
     \int_{x=0}^\ell \frac{dx}{2- \left(\ee{-x}+\ee{-(\ell-x)}\right)}
     =
     \int_{x=0}^\ell \frac{\ee{-x}dx}{2\ee{-x}- \left(\ee{-2x}+\ee{-\ell}\right)}
     \\&=
     \int_{x=0}^\ell \frac{\ee{-x}dx}{(1-\ee{-\ell})- (1-\ee{-x})^2}
     =
     \vast[\frac{\arctanh\Bigl(\frac{1-\ee{-x}}{\sqrt{1-e^{-\ell}}}\Bigr)}{\sqrt{1-e^{-\ell}}}
       \vast]_{x=0}^\ell
     \\&=
     \frac{\arctanh(\sqrt{1-e^{-\ell}})}{\sqrt{1-e^{-\ell}}} .
   \end{align*}
   Now to consider the asymptotic behaviour as $\ell\to \infty$, observe
   \begin{align*}
       \arctanh(z)&=\tfrac12 \ln\left(\frac{1+z}{1-z} \right)
       =
       \tfrac12 \ln\left(\frac{(1+z)^2}{1-z^2}\right)
       \\
       &=
       \ln(1+z)-\tfrac12\ln(1-z^2).
   \end{align*}
   Thus
   \begin{align*}
   \RMS(L_\ell)&=\frac{\ln(1+\sqrt{1-\ee{-\ell}})-\tfrac12\ln(1-(1-\ee{-\ell}))}{\sqrt{1-\ee{-\ell}}}
   \\
   &=\frac{\ln(1+\sqrt{1-\ee{-\ell}})+\frac{\ell}{2}}{\sqrt{1-\ee{-\ell}}},
   \end{align*}
   whence, as $1-\sqrt{1-\ee{-\ell}}$ decays exponentially to $1$, 
   \[\RMS(L_\ell)-(\ell/2 +\ln(2))\to 0\qquad\text{as }\ell\to\infty\]
   as required.
\end{proof}
This result should be compared with the magnitude for the interval of length $\ell>0$~\cite{LeinsterWillerton:AsymptoticMagnitude}:
  \[|L_\ell|= \ell/2 +1.\]
So asymptotically the magnitude and the spread of the interval have the same leading order term but different sub-leading terms.

As an aside, we can easily calculate the integral theoretic versions of $\RMSq2$ and $\RMSq\infty$ for the interval.  Again, asymptotically these have the the same leading order terms, but different sub-leading order terms.
\begin{thm}  For $\ell>0$ we have the following results for the length $\ell$ interval.
\begin{enumerate}
  \item $\RMSq2(L_\ell)= \frac{\ell^2}{2\ell-2(1-e^{-\ell})}$.
  \item $\RMSq2(L_\ell)-(\ell/2 +1/2)\to 0 $ as $\ell\to \infty$.
  \item $\RMSq\infty(L_\ell)= \frac{\ell}{2(1-e^{-\ell/2})}$.
  \item $\RMSq\infty(L_\ell)-\ell/2 \to 0 $ as $\ell\to \infty$.  
\end{enumerate} 
\end{thm}
\begin{proof}
  \begin{enumerate}
    \item This is obtained from the integral version of the order two spread:
    \begin{align*}
      \RMSq{2}(L_\ell)
      &=
      \frac{\bigl(\int_{x\in L_\ell} \mathrm{d}\mu(x)\bigr)^2}
      {\int_{x\in L_\ell} \int_{y\in L_\ell} e^{-\dd(x,y)} \,\mathrm{d}\mu(y)\,\mathrm{d}\mu(x)}
      \\
      &=
      \frac{\ell^2}{\int_{x=0}^{\ell}
              \left(2- \left(e^{-x}+e^{-(\ell-x)}\right)\right)\,\mathrm{d}\mu(x)}
              \\
      &=
      \frac{\ell^2}{2\ell-2(1-e^{-\ell})}.
    \end{align*}
    \item This follows from the above.
    \item This is obtained from the integral version of the order-infinity spread:
     \begin{align*}
      \RMSq{\infty}(L_\ell)
      &=\inf_{x\in L_\ell}
      \frac{\int_{y\in L_\ell} \mathrm{d}\mu(y)}
      {\int_{y\in L_\ell} e^{-\dd(x,y)} \,\mathrm{d}\mu(y)}
      =\inf_{x\in [0,\ell]}
      \frac{\ell}{
              2- (e^{-x}+e^{-(\ell-x)})}
              \\
      &=
      \frac{\ell}{2(1-e^{-\ell/2})}.
    \end{align*}
    \item This follows from the above.
    \end{enumerate}
\end{proof}

\subsection{Riemannian manifolds}
A Riemannian manifold is a smooth manifold equipped with a Riemannian metric, so in particular has an inner-product on each tangent space.  This structure gives rise to both a metric and a measure on the manifold.  The metric comes about because the Riemannian metric can be used to define a length for each rectifiable path in the manifold and the distance between two points is defined to be the infimum of the lengths of all the paths between the two points.  The measure comes about because the Riemannian metric can be used to define a volume form which leads to a density and a measure.  This means that every Riemannian manifold has a well-defined spread given by the formula
 \[\RMS(X):=\int_{x\in X} \frac{\mathrm{d}x}{\int_{y\in X} e^{-d(x,y)} \,\mathrm{d}y}.\]
In the case of homogeneous Riemannian manifolds this coincides with the formula for the magnitude that was examined in~\cite{Willerton:Homogeneous}.  In particular this tells us that the spread of $S^n_R$ the $n$-sphere of radius $R$ with its intrinsic metric, for $n\ge 1$ is given by
  \[\RMS(S^n_R)=\begin{cases}
\displaystyle\frac{2}{1+e^{-\pi R}} \prod_{i=1}^{n/2}\Bigl(\big(\tfrac{R}{2i-1}\big)^2 + 
1\Bigr) & \text{$n$ even}\\[1.5em]
\displaystyle\frac{\pi R}{ 1-e^{-\pi R}}\prod_{i=1}^{(n-1)/2}\Bigl(\bigl(\tfrac{R}{2i}\bigr)^2 + 
1\Bigr)\quad& 
\text{$n$ odd}.
\end{cases}\]
Moreover, the methods employed in~\cite{Willerton:Homogeneous} to calculate the asymptotics carry over essentially unchanged but work for \emph{all} closed Riemannian manifolds and not just homogeneous ones.
\begin{thm}
\label{Thm:RiemannianAsymptotics}
If $X$ is an $n$-dimensional Riemannian manifold (without boundary), with $\omega_n$ denoting the volume of the unit $n$-ball, $\Vol(X)$ denoting the volume of $X$ and $\TSC(X)$ 
denoting the total
scalar curvature of $X$ then as $X$ is scaled up the asymptotics of the 
spread are as follows:
  \[
   \RMS(tX)
    =
    \frac{1}{n!\,\omega_n}\Bigl(t^n\Vol(X) + \frac{n+1}{6}t^{n-2}\TSC(X)\Bigr.
     \Bigl.+O(t^{n-4})\Bigr)
    \quad \text{as }t\to\infty.
  \]
\end{thm}
\begin{proof}  This is almost identical to the proof of Theorem~11 in~\cite{Willerton:Homogeneous} except that now the scalar curvature is not a constant and should be written as $\tau(x)$. 
\end{proof}
This simplifies in the case $n=2$ as follows.
\begin{cor}
For $\Sigma$ a Riemannian surface, the spread is 
asymptotically given in terms of the area and the Euler characteristic by
\[\RMS(t\Sigma)=\frac{\Area(\Sigma)}{2\pi}t^2+\chi(\Sigma)+O(t^{-2})\quad 
\text{as }t\to\infty.\]
\end{cor}
\begin{proof}
  This follows from the theorem above as $\omega_2=\pi$ and the Gauss-Bonnet Theorem says that $\TSC(\Sigma)=4\pi\chi(\Sigma)$.
\end{proof}
\section{Dimension and fractals}
\label{Section:GrowthRate}
In this section we define the notion of spread dimension of a metric space which is the instantaneous growth rate of the spread of the space.  This notion of dimension is scale dependent.  For instance, we will see that a long, thin rectangular array of points can have spread dimension close to zero, one, or two, depending on the scale.  Then we look at the spread dimension of some finite approximations to simple fractals and see that the spread dimension is close to the Hausdorff dimension at some scales.  Finally we observe that Meckes has recently related the \emph{asymptotic magnitude} dimension of spaces to the Minkowski dimension, for the spaces considered here the Hausdorff and Minkowski dimensions are equal.

\subsection{Definition of spread dimension}
Now that we have a notion of size of a metric space, we can look at the growth rate of this size as a measure of the `dimension' of the space.  Typically one looks at the asymptotic growth rate as a measure of dimension, but it is interesting here to look at the instantaneous growth rate.  The size is very scale dependent in a non-obvious way so looking at how the growth rate varies is very interesting.  For a real-valued function $f$ defined on some subset of the reals, we can define the growth rate at $t$ by 
  \[(Gf)(t):=\frac{\mathrm{d}\ln(f(t))}{\mathrm{d}\ln(t)}.\]
For example, if $f(t)=t^n$ then $Gf(t)=n$.  Another way of writing this is as
  \[(Gf)(t):=\frac{t}{f(t)}\frac{\mathrm{d}f(t)}{\mathrm{d}t}.\]
The instantaneous growth rate is the gradient in a $\log$-$\log$ plot of the function.

We define $\Edim(X)$ the \defn{instantaneous spread dimension}, or just \defn{spread dimension}, of a metric space $X$ to be $Gf(1)$ where $f(t):=\RMS(tX)$, in other words,
  \begin{align*}
    \Edim(X)
    &:=
    \left.\frac{\mathrm{d}\ln(\RMS(tX))}{\mathrm{d}\ln(t)}\right|_{t=1}
    =
    \left.\frac{t}{\RMS(tX)}\frac{\mathrm{d}\RMS(tX)}{\mathrm{d}t}\right|_{t=1}.
  \end{align*}

It is then informative to look at examples of this instantaneous spread dimension as the space is scaled.   The following examples were calculated using \texttt{maple} on a processor with 16GB of RAM.

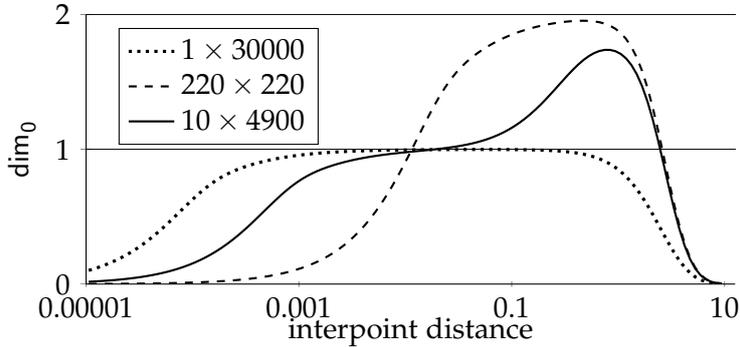
\begin{figure*}
\begin{center}
\beginpgfgraphicnamed{DimensionGraphI}
\begin{tikzpicture}
\begin{axis}[
width = 0.8\textwidth,
axis x line=bottom, axis y line = left,
xmin=-5,ymin=0, ymax=2, xmax=1.1,
xtick={-5,-3,-1,1,3}, xticklabels={$0.00001$,$0.001$,$0.1$,$10$,$1000$},
ytick={0,1,2}, yticklabels={0,1,2},
x axis line style={style = -},y axis line style={style = -},
xlabel={interpoint distance}, ylabel=$\Edim$,
yscale=0.5,
legend style={at={(0.05,1.9)},anchor=north west}
]
\addplot[mark=none,black,dotted,very thick] file {LineNoMatrixE0DerivN30000.dat};
\addlegendentry{$1\times 30000$};
\addplot[mark=none,black,dashed,thick] file {RectangleNotPlotE0Derivx220y220.dat};
\addlegendentry{$220\times 220$};
\addplot[mark=none,black,thick] file {RectangleNotPlotE0Derivx10y4900.dat};
\addlegendentry{$10\times 4900$};
\addplot[mark=none,very thin,black] expression {1};
\addplot[mark=none,very thin, black] expression {2};
\end{axis}
\end{tikzpicture}
\endpgfgraphicnamed
\end{center}
\caption{The spread dimension profiles for various rectangular grids.}
\label{Figure:RectangularGridDimensions}
\end{figure*}

\subsection{Rectangular grids}
As a first set of examples we can look at three types of rectangular grids with equally spaced points.  The spread-dimension profiles are shown in Figure~\ref{Figure:RectangularGridDimensions}.

Starting first with the grid of $1\times 30000$ points, or, in other words, a line of $30000$ points, we see that when the points are very close together the spread dimension is close to zero, reflecting the fact that the `line' at that scale is point-like.  As the line of points is scaled up, it looks more and more like a line, so when the interpoint distance is $0.01$ units, meaning the length is $300$ units, the spread dimension is close to one.  As the line of points is scaled up further and further, so that the interpoint distance is $10$ units, say, the point-like nature is apparent and the spread dimension drops to zero.

Considering the square grid of $220\times 220$ points, we see that this starts off looking like a point at small scales, with the spread dimension being close to zero, then as the square grid is scaled up to about $20$ units by $20$ units, with an interpoint distance of about $0.1$ units, it looks more like a genuine square and has an spread dimension of just under two.  Then as the square grid is scaled up further, the point-like nature is apparent and the spread dimension drops to zero.

The most interesting case shown is where we consider the rectangular grid of $10\times 4900$ points.  Again, at small scales the spread dimension is close to zero whilst the grid looks like a small point.  Then as it is scaled up there is a regime, around where the rectangle is of the order of $0.1$ units by $50$ units, where the space looks `line-like' and the dimension is approximately one.  As it is scaled up further to around $10$ units by $500$ units, the width is apparent and the spread dimension heads towards two.  Finally, as it is scaled up further, the point-like nature becomes apparent and the spread dimension descends to zero.

From this we deduce that the spread $\RMS$ must be measuring something geometric.

\subsection{Fractals}
We now look at the spread dimension of certain finite approximations to fractal sets in Euclidean space, namely to the ternary Cantor set, the Koch curve and the Sierpinski triangle.  We can look at the spread dimension profile and see that at certain scales the spread dimension is roughly the Hausdorff dimension of the corresponding fractal, indicating that spread is a reasonable measure of the size of these fractals, and, indeed, of these approximations to these fractals.

\begin{figure*}
\beginpgfgraphicnamed{DimensionGraphII}
\begin{tikzpicture}
\begin{axis}[
width = 0.47\textwidth,
axis x line=bottom, axis y line = left,
xmin=-3,ymin=0, ymax=2, xmax=6,
xtick={-5,-3,-1,1,3,5,7}, xticklabels={$10^{-5}$,$10^{-3}$,$10^{-1}$,$10^1$,$10^3$,$10^5$,$10^7$},
ytick={0,0.6309,1,2}, yticklabels={$0$,$\tfrac{\ln2}{\ln3}$,$1$,$2$},
x axis line style={style = -},y axis line style={style = -},
xlabel={length of Cantor set}, ylabel=$\Edim$,
legend style={at={(0.5,0.6)},anchor=south}
]
\addplot[mark=none,blue,dashed] file {CantorE0DerivDepth9N1024.dat};
\addlegendentry{$1024$ points};
\addplot[mark=none,red] file {CantorE0DerivDepth10N2048.dat};
\addlegendentry{$2048$ points};
\addplot[mark=none,black,very thin] expression[domain=-3:6] {0.6309};
\end{axis}
\end{tikzpicture}
\endpgfgraphicnamed
\beginpgfgraphicnamed{DimensionGraphIII}
\begin{tikzpicture}
\begin{axis}[
width = 0.47\textwidth,
axis x line=bottom, axis y line = left,
xmin=-0,ymin=0.619, ymax=0.641, xmax=5.1,
xtick={-5,-3,-1,1,3,5}, xticklabels={$10^{-5}$,$10^{-3}$,$10^{-1}$,$10^1$,$10^3$,$10^5$},
ytick={0.62,0.6309,0.64}, yticklabels={0.62,$\tfrac{\ln2}{\ln3}$,0.64},
x axis line style={style = -},y axis line style={style = -},
xlabel={length of Cantor set}, ylabel=$\Edim$,
legend style={at={(0.5,0.1)},anchor=south}
]
\addplot[mark=none,blue,dashed] file {CantorE0DerivDepth9N1024.dat};
\addlegendentry{$1024$ points};
\addplot[mark=none,red] file {CantorE0DerivDepth10N2048.dat};
\addlegendentry{$2048$ points};
\addplot[mark=none,black,very thin] expression {0.6309};
\end{axis}
\end{tikzpicture}
\endpgfgraphicnamed
\\
\beginpgfgraphicnamed{DimensionGraphIV}
\begin{tikzpicture}
\begin{axis}[
width = 0.48\textwidth,
axis x line=bottom, axis y line = left,
xmin=-3,ymin=0, ymax=2, xmax=4.5,
xtick={-5,-3,-1,1,3}, xticklabels={$0.00001$,$0.001$,$0.1$,$10$,$1000$},
ytick={0,1,1.262,2}, yticklabels={0,1,$\frac{\ln 4}{\ln 3}$,2},
x axis line style={style = -},y axis line style={style = -},
xlabel={width of Koch curve}, ylabel=$\Edim$,
legend style={at={(0.45,0.1)},anchor=south}
]
\addplot[mark=none,blue,dashed] file {KochNotPlotE0DerivDepth6N4097.dat};
\addlegendentry{$4097$ points};
\addplot[mark=none,red] file {KochNotPlotE0DerivDepth7N16385.dat};
\addlegendentry{$16385$ points};
\addplot[mark=none,black,very thin] expression {1.262};
\end{axis}
\end{tikzpicture}
\endpgfgraphicnamed
\beginpgfgraphicnamed{DimensionGraphV}
\begin{tikzpicture}
\begin{axis}[
width = 0.48\textwidth,
axis x line=bottom, axis y line = left,
xmin=-3,ymin=0, ymax=2, xmax=4,
xtick={-5,-3,-1,1,3}, xticklabels={$0.00001$,$0.001$,$0.1$,$10$,$1000$},
ytick={0,1,1.585,2}, yticklabels={0,1,$\frac{\ln 3}{\ln 2}$,2},
x axis line style={style = -},y axis line style={style = -},
xlabel={width of Sierpinski triangle}, ylabel=$\Edim$,
legend style={at={(0.45,0.1)},anchor=south}
]
%
\addplot[mark=none,blue,dashed] file {SierpinskiNotPlotE0DerivDepth8N9843.dat};
\addlegendentry{$9843$ points};
\addplot[mark=none,red] file {SierpinskiNotPlotE0DerivDepth9N29526.dat};
\addlegendentry{$29526$ points};
\addplot[mark=none,black,very thin] expression {1.585};
\end{axis}
\end{tikzpicture}
\endpgfgraphicnamed
\caption{The spread dimension profiles of finite approximations to certain fractals, compared to the Hausdorff dimensions of the fractals.}
\label{Fig:FractalDimensions}
\end{figure*}
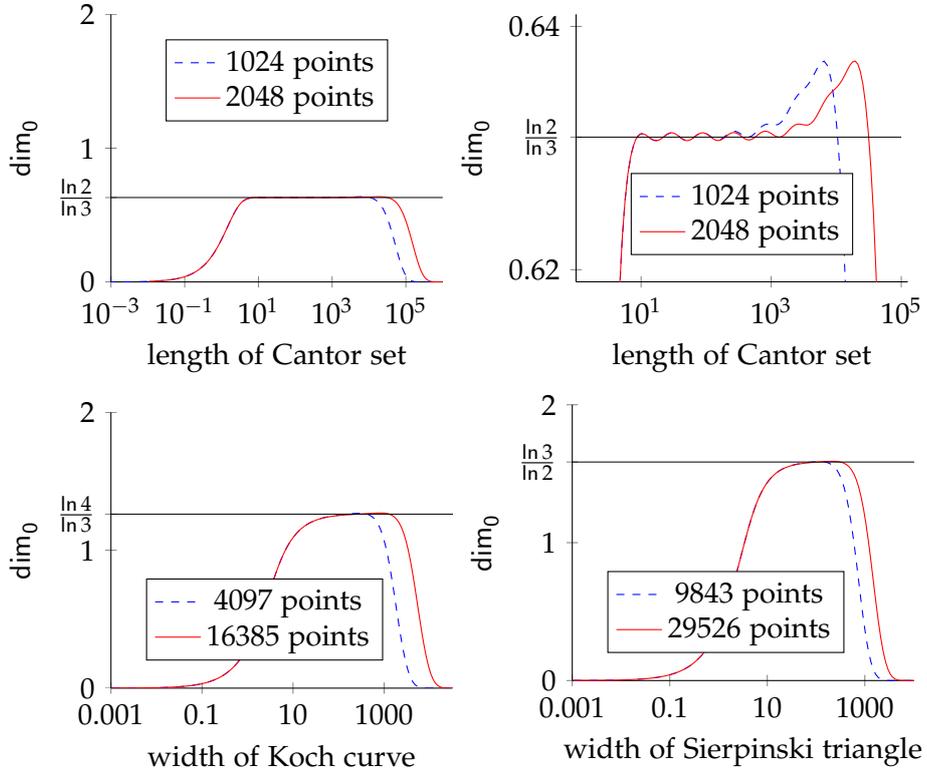

In the first case we look at the ternary Cantor set.  This is approximated by starting with two points a distance $\ell$ apart on a line.  We use the two contractions of the line by a factor of a third which respectively leave the two points fixed.  By applying these two contractions successively up to $10$ times, starting at the initial points, we obtain $2048$ points.  The spread dimension at various lengths $\ell$ can then be computed numerically.  In Figure~\ref{Fig:FractalDimensions} we see that at small scales the spread dimension is close to zero corresponding to the fact that the space looks like a point at those scales.  Similarly, at very large scales, the space looks like a collection of distant points and the spread dimension is again zero.  At intermediate scalings, roughly for $10<\ell<10000$ the spread dimension is roughly the Hausdorff dimension of the Cantor set, namely $\ln 2/ \ln 3$, indicating that the space looks more `Cantor set-like' at those scales.

The top-right picture in Figure~\ref{Fig:FractalDimensions} is an enlargement of the Cantor set profile, and shows that things are apparently more intriguing than one might guess.  At intermediate scales, the spread dimension seems to oscillate around the Hausdorff dimension, with the oscillations being of \emph{multiplicative} period $3$.  Such small oscillations were observed for the magnitude of the Cantor set in~\cite{LeinsterWillerton:AsymptoticMagnitude}.  I have no good explanation for these oscillations at the moment.

In the next case we look at the Koch curve.  Again, this is approximated by starting with a couple of points and iteratively applying one of four contractions, to obtain a finite metric space contained in the Koch curve.  The graph shows that as the approximation is scaled up from very small, the spread dimension increases to roughly the Hausdorff dimension, $\ln 4/\ln 3$, where it remains over a range of scales, before descending to zero as the approximating space is scaled up sufficiently so that its discrete, point-like nature is apparent.  

The final example of the Sierpinski triangle is generated in the same way using an iterated function system, and shows the same behaviour, namely, of having roughly the same spread dimension as its Hausdorff dimension at certain scales.

These examples should serve to show that there is something interesting going on which has yet to be examined fully.

\subsection{Asymptotic magnitude dimension and Minkowski dimension}
\label{section:MeckesMinkowski}
Here we have been considering the instantaneous spread dimension of spaces which is scale dependent.  You can also consider the asymptotic spread dimension which is scale \emph{independent}; for a space $X$  it is essentially the lim sup of the growth rate of $E_0(tX)$ as $t\to \infty$.  Following numerical and exact calculations in~\cite{Leinster:Magnitude,LeinsterWillerton:AsymptoticMagnitude, Willerton:Heuristics,Willerton:Homogeneous} Meckes showed~\cite{Meckes:MagnitudeEtc} that the asymptotic \emph{magnitude} of a space is defined if and only if the Minkowski dimension of the space is defined and in that case the two are equal.   (For the spaces considered above the Minkowski and Hausdorff dimensions agree.)  This further suggests that it is not far-fetched that the spread is encoding such geometric information.

\section*{Acknowledgements}
It is a pleasure to thank Tom Leinster and Mark Meckes for various helpful conversations, comments and terminological assistance.  Similarly I would like to thank the Centre de Recerca Matem\`atica at the Universitat Aut\`onoma de Barcelona where some of this work was carried out and where I had opportunity to talk about this work during the Exploratory Programme on the Mathematics of Biodiversity; I would also like to thank the participants of that programme for their input and enthusiasm.  Finally I thank Sam Marsh and Neil Dummigan for integral inspiration.

\end{document}